\documentclass{amsart}

\usepackage{style}

\title[]{Polarized K3 surfaces with an automorphism of order 3 and low Picard number}

\author{Dino Festi}

\address{Dipartimenti di matematica ``Tullio Levi-Civita'', Universit\`a degli studi di Padova, 35121, Padova, Italy.}
\email{dino.festi@unipd.it}

\date{\today}

\begin{document}

\begin{abstract}
    In this paper, for each $d>0$, we study the minimum integer $h_{3,2d}\in \mathbb{N}$ for which there exists a complex polarized K3 surface $(X,H)$ of degree $H^2=2d$ and Picard number $\rho (X):=\mathrm{rank} \Pic X = h_{3,2d}$ admitting an automorphism of order $3$. 
    We show that $h_{3,2}\in\{ 4,6\}$ and $h_{3,2d}=2$ for $d>1$.
    Analogously, we study the minimum integer   $h^*_{3,2d}\in \mathbb{N}$  for which there exists a complex polarized K3 surface $(X,H)$ as above plus the extra condition that the automorphism acts as the identity on the Picard lattice of $X$.
    We show that $h^*_{3,2d}$ is equal to $2$ if $d>1$ and equal to $6$ if $d=1$.
    We provide explicit examples of K3 surfaces defined over $\mathbb{Q}$ realizing these bounds.
\end{abstract}

\maketitle

\section{Introduction}

The study of automorphisms of K3 surfaces has seen a very intense activity in the last 40 years.
In~\cite{Nik76, Ste85} Nikulin and Stark proved that a group acting purely non-symplectically on an algebraic K3 is cyclic and finite.
More in particular, Nikulin proves that such a group can have order at most 66; if the group has prime order, then its maximal order is 19, \cite[Theorem 0.1.c), Corollary 3.2]{Nik80}.
In these notes we consider non-symplectic automorphisms of order $3$, a topic extensively treated in~\cite{AS08, Tak11}.
In particular, we focus on the interplay between the existence of non-symplectical automorphism of order $3$, a polarization of given degree, and the Picard number of the surface, as already done in~\cite{FNP22} for non-symplectic involutions.

More precisely, let $(X,H)$ denote a complex polarized K3 surface of degree $2d$, that is, $H$ is an ample divisor of $X$ and $H^2=2d$.
Assume that $X$ admits an automorphism $\alpha\in \Aut X$ of prime order $p$.
Then $\alpha$ induces an action $\alpha^*$ on $H^{2,0}(X)=\langle \omega \rangle$, and hence $\alpha^*=\zeta \omega$, with $\zeta^p=1$.
In this paper we focus on the case $p=3$.
In this case, if $\zeta=1$ then $\alpha$ is called \emph{symplectic} and  $\rho (X)\geq 13$, see \cite[\S 10]{Nik80};
if $\zeta\neq 1$, that is, $\zeta$ is a primitive $3$-rd root of unity  then  $\alpha$ is called \emph{non-symplectic} and 
$\rho (X)\geq 2$, see \cite{AS08, Tak11}.
As done in~\cite{FNP22}, one may ask when is this lower bound realized depending on the degree of the polarization.

\begin{definition}\label{d:cH32d}
Let 
$$
\cH_{3,2d}:=\{ (X,H,\alpha) \}
$$ 
denote the set of complex polarized K3 surfaces $(X,H)$ such that $H^2=2d>0$ and $X$ admits an automorphism $\alpha$ of  order $3$,
one can then define
$$
h_{3,2d}=\min_{X\in \cH_{3,2d} } \{ \rho (X) \}\; .
$$
\end{definition}

In this work we prove the following result.
\begin{theorem}\label{t:Main}
    If $d>1$, then $h_{3,2d}=2$. 
    For $d=1$, we have $h_{3,2d}=h_{3,2}\in\{4,6\}$.
\end{theorem}
To prove that $h_{3,2d}=2$ for every $d>1$ we will only consider \emph{non-symplectic} automorphisms because, as noted above, symplectic automorphisms force the Picard number to be higher than desired.
We first show that if $X$ is a K3 surface of degree $2d$ with a (non-symplectic) automorphism of order $3$, 
then  $\rho (X)\geq 2$.
Then we complete the proof by providing an explicit example of a K3 surface with an automorphism of order $3$, a polarization of degree $2d>2$ and Picard number~$2$.
To provide such example for every $d>1$, it is enough to consider a single K3 surface with an automoprhism of order $3$ and Picard lattice isometric to $U=\begin{pmatrix}
    0 & 1\\
    1 & 0
\end{pmatrix}$, because this lattice admits ample divisors of degree $2d$ for every $d>1$.

The above argument leaves open the case for $d=1$.
Indeed for $h_{3,2}$ we only have a partial answer.
The partiality of this result is due to the wide range of possibilities that arise when studying the Picard lattice of a K3 surface with a non-symplectic automorphism of order $3$ and a polarization of degree $2$.
In this case,  on the one hand it is easy to show that the Picard number has to be larger than $2$ and that there are examples of such surfaces with Picard number $6$; 
on the other hand, it is hard to control the ample cone of all the possible Picard lattices with rank $4$ and we could neither find in the literature nor construct examples of K3 surfaces of degree $2$ and Picard number $4$ admitting an automorphism of order $3$. See \autoref{r:Missingh32} for more details.

The impasse can be overcome if we allow for one extra hypothesis, in the spirit of~\cite{Tak11}: we assume that $\alpha$ acts as the identity on the Picard lattice.
As we will see in \autoref{s:SecondProof}, this is equivalent to considering `generic' K3 surfaces with the desired properties.
This extra assumption leads to the following definitions.
\begin{definition} We define
    $$
\cH_{3,2d}^*:=\{ (X,H,\alpha)\in \cH_{3,2d} \; | \; \alpha^*_{| \Pic X}=\mathrm{id}  \}
$$ 
and 
$$
  h^*_{3,2d}=\min_{X\in \cH_{3,2d}^* } \{ \rho (X) \}\; .
$$
\end{definition}
Clearly $h^*_{3,2d}\geq h_{3,2d}$.
We then prove the following result.
\begin{theorem}\label{t:MainStar}
    The following equalities hold:
    \[
    h^*_{3,2d}=
    \begin{cases}
        6 &\textrm{ if } d=1\, ,\\
        2 &\textrm{ if } d>1\, .
    \end{cases}
    \]
\end{theorem}
The first equality follows almost immediately from the first statement of~\autoref{t:Main};
the second equality builds upon the second statement: in this case we only have two possible lattices of rank $4$ and we show that none of them admits a polarization of degree $2$.
The crucial ingredient for the proofs of all the above results is the classification of the fixed locus of an automorphism of order $3$, provided by Artebani and Sarti in~\cite{AS08}, see \autoref{t:AS}.
\begin{remark}\label{r:OtherOrders}
    The problem treated in this paper naturally generalizes to  any prime order $p$:
    for any prime $p$ one can define $\cH_{p,2d}$ and $h_{p,2d}$ substitutig $3$ with $p$ in \autoref{d:cH32d}.
    In~\cite{Nik80}, Nikulin proves that if $\alpha$ is a non-symplectic automorphism of prime order $p$ on a K3 surface, then $p\leq 19$, see~\cite[Theorem 3.1(c)]{Nik80}.
    Using the classification of non-symplectic automorphisms of prime order by Artebani, Sarti, Taki~\cite{AST11}, one might try to compute $h_{p,2d}$ for every prime $p\leq 19$ and for every $d\geq 1$.
    This is indeed a joint work in progress with Wim Nijgh and Pablo Quezada Mora.
\end{remark}

The paper is structured as follows:
in~\autoref{s:K3with} we briefly review the background of complex K3 surfaces with an automorphism of order $3$;
in~\autoref{s:d2} we prove \autoref{t:Main};
finally, \autoref{t:MainStar} is proved in~\autoref{s:SecondProof}. 

\section*{Acknowledgments} 
While working on this subject I  benefited from conversations with Alice Garbagnati, Bert van Geemen and Remke Kloosterman.
I would also thank Bartosz Naskręcki, Wim Nijgh and Pablo Quezada Mora for their contribution to the final part of the work.
I am also grateful to the anonymous referee for pointing out a crucial mistake in the previous version of this paper.
The author was supported by the PRIN grant \texttt{PRIN202022AGARB\char`_01} while at Università degli Studi di Milano.

\section{Projective K3 surfaces with a non-symplectic automorphism of order 3}\label{s:K3with}
There are several works on complex K3 surfaces with an automorphism of order $3$, in these notes we will mostly use the results presented in~\cite{AS08}.

Let $X$ be a complex projective K3 surface and assume it admits an automorphism $\alpha\in \Aut X$ of order $3$.
Also assume that $\alpha$ is non-symplectic.
Hence $\alpha^3=1$ and $\alpha^*(\omega)=\zeta\omega$, where $\omega$ is the class generating $H^{2,0}(X)$ and $\zeta$ is a primitive third root of unity.
In what follows, $\zeta$ will always denote a primitive third root of unity.
Recall that if $L$ is a lattice, we denote by $L^*:=\Hom(L,\IZ)$ its \emph{dual lattice} and by 
$A_L=L^*/L$ its \emph{discriminant group}.
Notice that $\alpha$ induces an isometry of the lattice $H^2(X,\IZ)$, which we will denote by $\alpha^*$.
As $\Pic X$ can be viewed as a sublattice of $H^2(X,\IZ)$, we will denote by $\alpha^*$ also the isometry of $\Pic X$ induced by $\alpha$.

\begin{definition}
    We define $N(\alpha):=(H^2(X,\IZ))^{\alpha^*}$, the sublattice of $H^2(X,\IZ)$ fixed by $\alpha^*$.
\end{definition}

\begin{definition}
    Let $\cE=\IZ [\zeta]$ denote the ring of Eisenstein integers. A $\emph{$\cE$-lattice}$ is a couple $(L,\sigma)$ where $L$ is a lattice and $\sigma$ is a fixed-point-free isometry of order $3$ on $L$.
    If $\sigma$ acts as the identity on $A_L$, then $(L,\sigma)$ is called an \emph{$\cE^*$-lattice}.
\end{definition}

\begin{proposition}\label{p:Nalpha}
    Let $(X,\alpha)$ be a complex K3 surface with a non-symplectic automorphism of order~$3$.
    Then
    \begin{enumerate}
        \item $N(\alpha)$ is a primitive $3$-elementary sublattice of $\Pic X$;
        \item $(N(\alpha)^\perp, \alpha^*)$ is a $\cE^*$-lattice, where $N(\alpha)^\perp$ is the orthogonal complement of $N(\alpha)$ inside $H^2(X,\IZ)$;
        \item $(T_X,\alpha^*)$ is a $\cE$-sublattice of $N(\alpha)^\perp$, where $T_X$ denotes the transcendental lattice of $X$.
    \end{enumerate}
\end{proposition}
\begin{proof}
  This is the reformulation of~\cite[Theorem 0.1]{Nik76} and~\cite[Lemma 1.1]{MO98} as in~\cite[Theorem 1.4]{AS08}.
\end{proof}

\begin{lemma}\label{l:Elattice}
    The following statements hold:
    \begin{enumerate}
        \item Any $\cE$-lattice has even rank;
        \item Any $\cE^*$-lattice is $3$-elementary.
    \end{enumerate}
\end{lemma}
\begin{proof}
    This is~\cite[Lemma 1.3]{AS08}.
\end{proof}

\begin{corollary}\label{c:evenrho}
    Let $(X,\alpha)$ be a complex K3 surface with a non-symplectic automorphism of order $3$.
    Then $\rho (X)$ and $\rank N(\alpha)$ are even.
\end{corollary}
\begin{proof}
    By~\autoref{p:Nalpha} we have that $(T_X,\alpha^*)$ is a $\cE$-lattice.
    Then from~\autoref{l:Elattice} it follows that $\rank T_X$ is even.
    As $\rho (X) = 22-\rank T_X$, we conclude the argument. 

    The same argument applied to $N(\alpha)^\perp$ shows that $\rank N(\alpha)$ is also even.
\end{proof}

The main result of~\cite{AS08} is the complete classification of the K3 surfaces $(X,\alpha)$ in terms of the fixed loci $\Fix \alpha\subset X$ and $N(\alpha)\subset \Pic X$.
Moreover, for each case they also provide a projective model realizing it.
Their results can be summarized as follows.

\begin{theorem}{\cite[Theorems 2.8 and 3.4]{AS08}}\label{t:AS}
    Let $(X,\alpha)$ be a complex K3 surface with an automorphism $\alpha$ of order $3$. 
    Then $\Fix \alpha$ consists of $n\leq 9$ points and $k\leq 6$ curves.
    The couple $(n,k)$ uniquely determines $N(\alpha)$.
    All the possible triples $(n,k,N(\alpha))$ are listed in~\cite[Table 2]{AS08}.

    Conversely, for every triple $(n,k,N(n,k))$  in~\cite[Table 2]{AS08} there exists a complex K3 surface $X_{n,k}$ with a non-symplectic automorphism $\alpha$ of order $3$ such that $\Fix \alpha$ consists of $n$ points and $k$ curves and $\Pic X_{n,k} = N(\alpha) \cong N(n,k)$.
    For each triple $(n,k,N(n,k))$ a projective model of such $X_{n,k}$ is given.
\end{theorem}

As we are only interested in K3 surfaces with low Picard number,
in \autoref{t:PossibleLEQ6} we only include the first entries of \cite[Table 2]{AS08}, omitting the transcendental lattice and indicating the type of projective model provided by Artebani and Sarti.

\begin{table}[h!]
\centering
\begin{tabular}{|c | c | c | l|}
 \hline
 $n$ & $k$ & $N$ & model for $X_{n,k}$ \\ [0.5ex] 
 \hline
 0 & 1 & $U(3)$ & Quadric $\cap$ cubic $\subset \IP^4$ \\ 
 ~ & 2 & $U$    & Weierstrass model\\
 \hline
 1 & 1 & $U(3)\oplus A_2$ & Quartic in $\IP^3$ \\
 ~ & 2 & $U\oplus A_2$    & Weierstrass model \\
 \hline
 2 & 1 & $U(3)\oplus A_2^{\oplus 2}$ & Double cover of $\IP^2$ \\
 ~ & 2 & $U\oplus A_2^{\oplus 2}$    & Weierstrass model\\ 
 \hline
\end{tabular}
\caption{Table of possible cases of $(n,k,N(\alpha))$ for $(X,\alpha)$ with $\rank N(\alpha)\leq 6$. In the last column we indicate the type of projective model provided in~\cite{AS08}.}
\label{t:PossibleLEQ6}
\end{table}

This result is very convenient because it tells us where to look in order to find polarized K3 surfaces of any degree admitting an automorphism of order $3$, as shown in the following sections.

\begin{remark}\label{r:Subfamily}
    The K3 surfaces with a given marking  and an automorphism of order $3$ form a \emph{subfamily} of K3 surfaces with the same marking.
    To see this, let $(X,\alpha)$ be a very generic complex K3 surface with  a non-symplectic automorphism $\alpha$ of order $3$.

    Let $V$ denote the $\IC$-vector space given by $H^2(X,\IZ)\otimes \IC$.
    Then $\alpha^*$ acts on $V$ and its action induces an orthogonal decomposition of $V$ into eigenspaces:
    $$
    V=V_1\oplus V_\zeta \oplus V_{\zeta^2}\; .
    $$
    As $\alpha$ is non-symplectic we can assume that $H^{0,2}(X)\subseteq V_{\zeta^2}$.
    We know that $N(\alpha )\subseteq \Pic X$ by~\autoref{p:Nalpha}.(1); 
    as we assumed $X$ to be very generic, we have that $\Pic X=N(\alpha)=V_1$ and
    hence $T_X\otimes \IC = V_\zeta \oplus V_{\zeta^2}$.
    As $V_\zeta$ and $V_{\zeta^2}$ are swapped by $\alpha^*$, they have the same dimension, and so we conclude that 
    $$
    \rank T_X = \dim (T_X\otimes \IC) = 2\dim V_\zeta.
    $$
    This means that if $(X,\alpha)$ is a K3 surface with an automorphism of order $3$, its period $\omega$ lies in $\IP (V_\zeta)$ which has dimension 
    $$
    \dim \IP (V_\zeta) = (\rank T_X)/2 -1 = 10-\rho (X)/2.
    $$
    
    On the other hand, if we just consider a marked K3 surface $X$ with $\Pic X\cong L$, without any other assumption, 
    then 
    the period of $X$ will lie in 
    $$
    \{ \omega \in \IP (T_X\otimes \IC) \; : \; \langle \omega, \omega \rangle = 0 \, , \, \langle \omega, \Bar{\omega} \rangle > 0\; \},
    $$
    which has dimension $\rank T_X-2=20-\rho (X)$.
\end{remark}

\section{Proof of the first theorem}\label{s:d2}
Throughout this section, 
let $X$ be the complex projective elliptic K3 surface defined by
\begin{equation}\label{eq:U}
  X\colon y^2=x^3+p(t)\; ,
\end{equation}
with $p(t)$ a polynomial of degree $12$  with only simple roots.
It is easy to see that the map $\alpha\colon (x,y,t)\mapsto (\zeta x, y, t)$ defines an automorphism of order $3$ on $X$.
We also assume $p(t)$, and hence $X$, to be generic.

\begin{lemma}\label{l:X32d}
    \begin{enumerate}
        \item $X$ admits an elliptic fibration;
        \item $\Pic X= \langle F, O \rangle \cong U$, where $F$ is the class of the fiber of the elliptic fibration and $O$ is the class of the unique section.
        \item $\alpha^*(F)=F$ and $\alpha^*(O)=O$, that is, $N(\alpha) = \Pic X$.
    \end{enumerate}
\end{lemma}
\begin{proof}
    The existence of the elliptic fibration is immediate from~\eqref{eq:U}. 
    It is also easy to see that the Gram matrix of $\langle  F, O \rangle$ is 
    \[\begin{pmatrix}
        0 & 1 \\
        1 & -2
    \end{pmatrix},\] 
    and hence $U\cong \langle  F,O \rangle\subseteq \Pic X$.
    As $X$ is assumed to be generic, we also have that $\langle  F,O \rangle = \Pic X$, see~\cite[Proposition 4.2, Theorem 5.6]{AS08}.
    The third statement is then just an immediate consequence of the second, as by~\cite[Proposition 4.2]{AS08} we have that $N(\alpha) \cong U \cong \Pic X$.
\end{proof}

\begin{proposition}\label{p:AmpleConeU}
    Let $e,f$ be the generators of $U\cong \Pic X$ such that $e^2=f^2=0$ and $e.f=1$.
    Then, up to a choice of signs, the ample cone of $\Pic X$ is given by the divisors $D=xe+yf$ such that $y>x>0$.
\end{proposition}
\begin{proof}
    Notice that in $U$ there are only two $-2$-classes: $\pm (e-f)$.
    Assume $O=(e-f)$ is effective. Hence $O$ is the only effective $-2$-curve of $S$.
    The positive cone of $X$ is given by divisors $xe+yf$ such that $xy>0$.
    Hence the ample cone is given by divisors $D=xe+yf$ such that $xy>0$ and $D.(e-f)>0$. As $D.(e-f)=-x+y$, we obtain the desired statement.
\end{proof}


We are now ready to prove the first theorem.
\begin{proof}[Proof of \autoref{t:Main}]
    To prove the first statement,
    consider $X$ defined in \eqref{eq:U}.
    Let $e,f$ denote two generators of $U\cong \Pic X$ as in \autoref{p:AmpleConeU}, 
     and consider the divisor $D=e+df$.
    Then, by~\autoref{p:AmpleConeU}, $D\in \Pic X$ is ample because $d>1$.
    As $D^2=2d$ and $X$ has an automorphism of order $3$ then $X\in\cH_{3,2d}$ and hence 
 $h_{3,2d}\leq 2$.
    As in general $h_{3,2d}\geq 2$ (it follows immediately from~\autoref{t:AS}), we conclude that $h_{3,2d}=2$.

    To prove the second statement, 
    consider $(Y,H,\alpha)\in \cH_{3,2}$ and let $N(\alpha)\subset \Pic Y$ be the fixed locus of $\alpha^*$.
    As the the fixed locus of $\alpha$ is not empty~\cite[Theorem 2.2]{AS08} we have that $\rho (Y) \geq \rank N(\alpha)\geq 2$.
    If $\rho (Y)=2$ it follows that $\Pic Y = N(\alpha)= U$ or $U(3)$ (\autoref{t:AS}).
    In both cases $\Pic Y$ does not admit an ample divisor of degree $2$, hence a contradiction with the hypothesis that $H^2=2$.
    From this it follows that $h_{3,2}>2$.

    On the other hand, from~\autoref{t:PossibleLEQ6} we see that there exists a complex K3 surface $Y$ with an automorphism $\sigma$ of order $3$ and $\Pic Y=N(\sigma) \cong U(3)\oplus A_2^{\oplus 2}$ 
    and this $Y$ can be realized as double cover of $\IP^2$, i.e., it admits a polarization of degree $2$.
    From this it follows that $h_{3,2}\leq 6$.

    As the Picard lattice of a K3 surface with a non-symplectic  automorphism of order $3$ is always even (\autoref{c:evenrho}), we conclude that $h_{3,2}\in \{ 4,6 \}$, proving the second statement.
\end{proof}

\begin{remark}\label{r:U}
    We are left to show that, for every integer $d>1$, the bound $h_{3,2d}=2$ can be attained over $\IQ$.
    In fact, a priori $X$ is only defined over $\IC$.
    One might expect that for a random choice of rational coefficients of $p(t)$ one still obtains a K3 surface with Picard number $2$.
    A practical problem arises though: computing the Picard number of a  K3 surface given in its Weierstrass form is not easy.

    Luckily we can use the beautiful K3 surface
    $$
      X_{66}\colon y^2=x^3+t(t^{11}-1)
    $$
    considered by Kondo in~\cite{Kon92}, that is defined over $\IQ$.
    This surface has the remarkable property of being the \emph{unique} K3 surfaces (up to isomorphism) to admit a $\IZ/66\IZ$-action, hence admitting an automorphism of order $3$.
    The Picard lattice of $X_{66}$ is indeed $U$~\cite[Example 3.0.1]{Kon92}.
\end{remark}

\begin{remark}\label{r:Missingh32}
    Assume $(Y,H,\alpha)$ is in $\cH_{3,2}$.
    As we will see in the next section, if 
    \[
        \Pic Y=N(\alpha)
    \] 
    then $\rho (Y)\geq 6$.
    Unfortunately, the above equality does not need to always hold, for instance see \autoref{e:NonGeneric}.
    Therefore at the moment we cannot exclude  that $\rho (Y)=4$.
    In this case, $\Pic Y$ is a finite index overlattice of $N(\alpha)\oplus L$, where $L$ is the orthogonal complement of $N(\alpha)$ in $\Pic Y$.
    From \autoref{t:PossibleLEQ6}, we know that $N(\alpha)$ is either $U$ or $U(3)$.
    Notice that if $N(\alpha) = U$, then $\Pic Y = N(\alpha)\oplus L$, see~\cite[Example 14.0.6]{Huy16}.
    Moreover, as $L$ is not contained in $N(\alpha)$, the isometry $\alpha^*$ of $\Pic Y$ induces an isometry of order $3$ on $L$.
    Hence $L$ is a negative definite lattice of rank $2$ with an isometry of order $3$:
    by \cite[Lemma 6.11]{LZ22}, it follows that $L\cong A_2(j)$ for some $j\geq 1$.
    This means that $\Pic Y$ is either $U\oplus A_2(j)$ or an overlattice of $U(3)\oplus A_2(j)$, for some value of $j\geq 1$.
    It is possible to show that for some values of $j$, the above lattices do not admit $2$-divisors, but we have been unable to reach a general understanding of the existence of ample $2$-divisors for all the values of $j$.
\end{remark}

\begin{example}\label{e:NonGeneric}
    It is not hard to construct an example of a K3 surface with an automorphism of order $3$ acting non-trivially on the Picard lattice.
    For example, consider an elliptic K3 surface $Z$ defined 
    as in~\eqref{eq:U}, with $p(t)$ with only simple roots and equal to $f_6^2-g_4^3$, for some $f_6$ and $g_4$ polynomials in $t$ of degree $6$ and $4$, respectively.
    Then $N(\alpha)\cong U$ and one
    can easily check that $Z$ has at least two sections, namely $(g_4,f_6,t)$ and $(\zeta_3g_4,f_6,t)$.
    We then conclude that $\rho (Z)\geq 4$ and $N(\alpha) = U\subsetneqq \Pic Z$.
\end{example}

\section{The proof of the second theorem}\label{s:SecondProof}

In this section assume that $(X,H,\alpha)$ is in $\cH^*_{3,2d}$,
that is, $X$ is a projective K3 surface with a polarization $H$ of degree $2d$ and an automorphism $\alpha$ of order $p$ acting as the identity on the whole $\Pic X$. 
Using the notation introduced in \autoref{s:K3with}, this means that $N(\alpha)= \Pic X$. 
It also means that $(X,\alpha)$ is generic in the moduli space of K3 surface with an automorphism $\sigma$ of order $3$ and fixed locus of $\sigma^*$ equal to $N(\alpha)$, see~\cite[Theorem 5.6]{AS08}.
The classification of the fixed locus of an order $3$ non-symplectic automorphism given in~\cite{AS08, Tak11} is the key in establishing $h^*_{3,2}$.

\begin{lemma}\label{l:h32dstar}
    For every $d>1$, one has $h^*_{3,2d}=2$.
\end{lemma}
\begin{proof}
    Recalling that $h^*_{3,2d}\geq h_{3,2d}$, \autoref{t:Main} implies that $h^*_{3,2d}\geq 2$.
    On the other hand, \autoref{l:X32d} shows that the surface defined in~\eqref{eq:U} is in $\cH^*_{3,2d}$ and has Picard number $2$, concluding the proof.
\end{proof}

\begin{remark}
    Kondo's surface mentioned in \autoref{r:U} shows that also the bound $h^*_{3,2d}$ can be attained over $\IQ$, for every $d>1$.
\end{remark}

We are left with case for $d=1$.

\begin{lemma}\label{l:hStar4}
    If $X\in\cH^*_{3,2d}$ has Picard number $4$, then $d>1$.
\end{lemma}
\begin{proof}
    From the hypothesis, using~\autoref{t:PossibleLEQ6}, it immediately follows that $\Pic X$ is either $U\oplus A_2$ or $U(3) \oplus A_2$.

    First we show that $U(3)\oplus A_2$ does not admit $2$-divisors at all.
    Let $u_1,u_2$ and $a_1,a_2$ denote the generators of $U(3)$ and $A_2$, respectively and let 
    $$
    D:=x_1u_1+x_2u_2+y_1a_1+y_2a_2
    $$
    a $2$-divisor,
    that is, $D^2=2$.
    Dividing by $2$, we get the following equality:
    \begin{equation}
        3x_1x_2-y_1^2+y_1y_2-y_2^2=1 \; .
    \end{equation}
    This can be rewritten as 
    \begin{equation}\label{eq:U3}
        3x_1x_2-1 = y_1^2-y_1y_2+y_2^2.
    \end{equation}
    Reducing modulo $3$, \eqref{eq:U3} induces the following equation:
    \begin{equation}\label{eq:U3mod3}
        y_1^2-y_1y_2+y_2^2 \equiv 2 \mod 3\; .
    \end{equation}
    It easy to see by direct computations that \eqref{eq:U3mod3} has no solutions in $\IZ/3\IZ$, proving the claim.

    Assume then that $\Pic X \cong U\oplus A_2$. This implies that $U\hookrightarrow N(\alpha) = \Pic  X$ and hence $X$ is elliptic and can be described by the following Weierstrass equation \cite[Proposition 4.2]{AS08}:
    $$
    y^2=x^3+p(t),
    $$
    where $p$ is a polynomial of degree $12$.
    As $\Pic X \cong U\oplus A_2$ we see that $X$ has only one reducible singular fiber, of Kodaira type IV. 
    This implies that $\Pic X$ is generated by the class of the fiber $F$, the class of the section $O$ and the two components $E_1, E_2$ of the reducible fiber not meeting $O$.
    Using these four generators, the Gram matrix of $\Pic X$ is the following.
    $$
    \begin{pmatrix}
        0 & 1 & 0 & 0\\
        1 & -2 & 0 & 0\\
        0 & 0 & -2 & 1\\
        0 & 0 & 1 & -2
    \end{pmatrix}
    $$
    Let us write $H=aO+fF+e_1E_1+e_2E_2$ and notice that the third component of the singular fiber can be written as $E_3=F-E_1-E_2$.
    Then $H^2=2$ implies 
    \begin{equation}\label{eq:UA2}
        af+e_1e_2=e_1^2+e_2^2+a^2+1\; .
    \end{equation}
    As $H$ is ample, its intersection with all the $-2$-curves is positive, that is, 
    $$
    \begin{cases}
        H.O&=f-2a>0\; ,\\
        H.E_1 &= -2e_1+e_2 >0\; ,\\
        H.E_2 &= e_1-2e_2 >0\; , \\
        H.E_3 &= a+e_1+e_2 >0\; .
    \end{cases}
    $$
    From the above inequalities we deduce that
    $$
    \begin{cases}
        0<2a < f\; ,\\
        e_1, e_2 <0 \; ,\\
        0<-e_1-e_2<a \; .
    \end{cases}
    $$
    Consider then the quantity $2a^2+1$.
    As $2a<f$ and $e_1e_2\geq 1$ we can write 
    $$
    2a^2+1 < af +e_1e_2 < af +3e_1e_2 \; .
    $$
    Using~\eqref{eq:UA2}, we can substitute $af+e_1e_2$, hence obtaining
    $$
    2a^2+1<e_1^2+e_2^2+a^2+1+2e_1e_2 = (-e_1-e_2)^2+a^2+1 < 2a^2+1 
    $$
    as $-e_1-e_2$ is strictly smaller than $a$.
    In this way we get 
    $$
    2a^2+1< 2a^2+1\; ,
    $$
    a contradiction, proving that $U\oplus A_2$ admits no ample $2$-divisors.
    This concludes the proof.
\end{proof}

\begin{lemma}\label{l:h32star}
    $h^*_{3,2}=6$.
\end{lemma}
\begin{proof}
    As $h^*_{3,2}\geq h_{3,2}$, from \autoref{t:Main} it follows that $h^*_{3,2}>2$.
    From \autoref{l:hStar4} we also know that $h^*_{3,2}\neq 4$, hence $h^*_{3,2}\geq 6$.
    As already noted in the proof of the second statement of \autoref{t:Main}, from~\autoref{t:PossibleLEQ6} we see that there exists a complex K3 surface $Y$ with an automorphism $\sigma$ of order $3$ and $\Pic Y=N(\sigma) \cong U(3)\oplus A_2^{\oplus 2}$ 
    and this $Y$ can be realized as double cover of $\IP^2$.
    From this it follows that $h^*_{3,2}= 6$.
\end{proof}
Now we have everything we need to prove the second theorem.
\begin{proof}[Proof of~\autoref{t:MainStar}]
The first equality is \autoref{l:h32star};
the second is \autoref{l:h32dstar}.
\end{proof}

\begin{remark}\label{r:6Existence}
    It is easy to show that the bound $h^*_{3,2}$ can be attained over $\IQ$.
    Indeed the paper~\cite{AS08} tells us how to find explicit examples of K3 surfaces of degree $2$ with Picard number equal to $6$,
    just by considering a surface as  \cite[Proposition 4.11]{AS08} that is generic enough.
    For example, consider the K3 surface $X_{2,1}$ given by the double cover of $\IP^2$ branched along the curve
    $$
      B\colon F_6(x_0,x_1)+F_3(x_0,x_1)x_2^3+bx_2^6
    $$
    with
    \begin{align*}
      F_6 &:=-x_0^6 + 2x_0^5x_1 - x_0^4x_1^2 - 2x_0^3x_1^3 - x_0^2x_1^4 + x_0x_1^5 - x_1^6\; ,\\
      F_3 &:=2x_0^2x_1 - x_1^3\; ,\\
      b &:= 2\; .
    \end{align*}
    From~\cite[Proposition 4.11]{AS08} we know that $\rho (X_{2,1})\geq 6$.
    By reducing modulo a prime of good reduction, e.g. 11, one can see that $\rho (X_{2,1})=6$ and hence $\Pic X_{2,1}\cong U\oplus A_2^{\oplus 2}$.
\end{remark}

\bibliographystyle{amsplain}
\bibliography{references}

\end{document}